\documentclass[reqno]{amsart}
\usepackage{amsthm, amsfonts, amssymb, color}
\usepackage{mathrsfs}
\usepackage{geometry}
\usepackage{caption}
\usepackage{enumerate}
\usepackage{cases}
\usepackage{amsmath}
\usepackage{mathbbol}
\usepackage{footnote}
\usepackage{amstext, amsxtra}
\usepackage{txfonts}
\usepackage{tikz}
\usepackage{enumitem}
\usepackage[colorlinks, linkcolor=black, citecolor=blue,   hypertexnames=false]{hyperref}
\usepackage{graphicx}
\usepackage{subfigure}
\usepackage[symbol]{footmisc}
\allowdisplaybreaks
\def\R {\mathbb{R}}

\def\x{\boldsymbol{x}}

\def\d{{\rm d}}

\def \and {{\qquad\text{and}\qquad}}

\newtheorem{theorem}{Theorem}[section]

\newtheorem{lemma}[theorem]{Lemma}
\newtheorem{definition}[theorem]{Definition}

\newtheorem{remark}[theorem]{Remark}

\numberwithin{equation}{section}
\theoremstyle{definition}

\title[]
{Fractal Tur\'{a}n--Nazarov Inequality and Observability for  Schr\"{o}dinger Equations}

\author{Jiaqi Yu \quad Shanlin Huang }

\address{Jiaqi Yu,  School of Mathematics and Statistics, Huazhong University of Science and Technology,  Wuhan,  430074,  P.R. China}
\email{jiaqi\_yu@hust.edu.cn}
\address{Shanlin Huang,  School of Mathematics (Zhuhai), Sun Yat-sen University, Zhuhai 519082, Guangdong, China}
\email{huangshlin6@mail.sysu.edu.cn}


\begin{document}
	
	\begin{abstract}
		This paper establishes limitations on observability inequality and unique continuation for Schr\"{o}dinger equations on fractal sets. We prove that,  in contrast to the heat equation, such properties can fail in fractal settings.
		To achieve this, we first extend the classical  Tur\'{a}n--Nazarov inequality, which provides lower bounds of trigonometric polynomials of the form $\sum_{k=1}^nc_ke^{2\pi im_kt}$ on sets of positive measure, to the fractal setting. Unlike in the classical case,  the constant in the inequality loses uniformity in the degree $n$,  and we obtain sharp bounds depending on both $n$ and the frequency difference $m_n-m_1$. 
		These refinements  then enable us to construct explicit counterexamples, showing that observability and unique continuation may fail for Schr\"{o}dinger equations when the observation set is fractal. 
		
	\end{abstract}
	
	\maketitle
	

	\section{Introduction and main results}\label{section1}
	
	Inequalities that control a function's values on a set by its values on a smaller subset are known by various names, including Remez-type inequalities, propagation of smallness, or  quantitative unique continuation. Recently, such estimates have become important in control theory, particularly for spectral subspaces of the Laplace operator, see for instance in \cite{ApraizNullcontrol2013,BM,EV,Lebeau1995contr,N,wang2019,WWZ} and references therein.
	
	In the simplest case, where the domain is the interval $[0,1]$, the Laplacian eigenfunctions are complex exponentials, in whcih case the functions of interest are trigonometric polynomials
	\begin{equation}\label{eq-tri-01}
		p(t)=\sum_{k=1}^nc_ke^{2\pi im_kt}, \quad c_k \in \mathbb{C},
	\end{equation}
	with distinct frequencies
	\begin{equation}\label{eq-fre-01}
		m_1<\cdots<m_n, \quad \quad m_k\in \mathbb{Z},\,\,1\le k\le n.
	\end{equation}
	We are motivated by the celebrated propagation of smallness result, called the  Tur\'{a}n lemma \cite{Tu}, which  states that for any subinterval  $E \subset [0, 1]$, 
	\begin{equation}\label{eq-Turan}
		\sup_{t\in[0,1]}|p(t)|\leq \left(\frac{4e}{|E|}\right)^{n-1}\sup_{t\in E}|p(t)|,
	\end{equation}
	where $|E|$ represents the Lebesgue measure of  $E$. A major improvement is due to  Nazarov \cite{N} where $E$ can be  arbitrary subsets of positive measure, with the constant $4e$ in \eqref{eq-Turan} replaced by 14.
	The noteworthy fact is that the constant depends only on the number of frequencies $n$, and not on the specific values of the frequencies $m_k$ ($1\le k\le n$).
	This strengthened  inequality plays a crucial role in Nazarov's proof of the sharp Amrein-Berthier uncertatinty principle 
	\begin{equation}\label{eq-AB-un}
		\|f\|_{L^{2}(\mathbb{R})}^{2} \leq c_1e^{c_2|E||F|}\left(\|f\|_{L^{2}(\mathbb{R}\setminus E)}^{2}+\|\hat f\|_{L^{2}(\mathbb{R}\setminus F)}^{2}\right),
	\end{equation}
	where $E, F$ are of finite Lebesgue measure and $c_1, c_2>0$ are absolute constants.
	
	The Tur\'{a}n--Nazarov inequality has  been extended in several directions. For example, Brudnyi \cite{Br} established a version for quasipolynomials in one or several variables.
	Friedland and Yomdin \cite{FY} introduced a geometric refinement by replacing the Lebesgue measure  of $E \subset [0, 1]$ with a more sophisticated geometric invariant
	$\omega_f(E)$, called the \emph{metric-span} of $E$, which depends on the function $f$.
	Moreover, Fontes-Merz \cite{FM} provided a multidimensional extension of the Tur\'{a}n--Nazarov inequality by induction.
	
	It is natural to ask whether the Tur\'{a}n–Nazarov inequality remains valid when $E$ is a fractal set.
	Our objectives in this paper are twofold:
	\begin{enumerate}
		\item We investigate the validity of Tur\'{a}n--Nazarov inequality when the Hausdorff dimension of $E$ in \eqref{eq-Turan}  is strictly less than 1.
		
		\item  We apply the resulting estimates to observability and unique continuation inequalities for Schr\"{o}dinger equations on such sets.
	\end{enumerate}
	
	To state the result, we	recall the following definition.
	\begin{definition}
		For $\alpha> 0$, the $\alpha$-Hausdorff content of a set $E\subset\R$ is defined as 
		\begin{align}\label{eq-content}
			C^{\alpha}_{\mathcal{H}}(E)=\inf \bigg\{\sum_{j} r_j^\alpha:\,E\subset \bigcup_{j} B(x_j, r_j)\bigg\}.
		\end{align}
		The Hausdorff dimension of $E$ is then given by
		$
		\dim_{\mathcal{H}}(E) = \inf \{ \alpha > 0 \mid C^{\alpha}_{\mathcal{H}}(E) = 0 \}.
		$
	\end{definition}
	
	In this work, we employ the $\alpha$-Hausdorff content defined in \eqref{eq-content} rather than the more conventional $\alpha$-Hausdorff measure. The latter is defined for $E \subset \mathbb{R}$ as
	$$
	\mathcal{M}^{\alpha}_{\mathcal{H}}(E) = \lim_{\delta \to 0} \inf \left\{ \sum_j \operatorname{diam}(A_j)^{\alpha}:\, E \subset \bigcup_j A_j, \, \operatorname{diam}(A_j) \leq \delta \right\}.
	$$
	While the $s$-Hausdorff content and measure may differ in value, both definitions yield identical Hausdorff dimensions (see e.g. in \cite[Lemma 4.6]{Ma}).
	

	\begin{theorem}\label{thm1}
		Let $\alpha\in (0,1)$ be given.
		
		$(i)$  There are no analogs of  \eqref{eq-Turan} of the form
		\begin{equation}\label{equ3.1}
			\sup_{t\in[0,1]}|p(t)|\leq C\sup_{t\in E}|p(t)|
		\end{equation}
		that holds for all trigonometric polynomials of degree $n$ and for all sets  $E\subset[0,1]$ of  $C^{\alpha}_{\mathcal{H}}(E)>0$, with a constant $C$ depending  only on $E$ and $n$.
		
		$(ii)$ For any  $E\subset[0,1]$ with  $C^{\alpha}_{\mathcal{H}}(E)>0$, and for every  trigonometric polynomial $p(t)$ of the form  \eqref{eq-tri-01}-\eqref{eq-fre-01}, there exists some absolute constant $C_0>0$  such that
		\begin{equation}\label{eq3.1}
			\sup_{t\in[0,1]}|p(t)|\leq \Bigg(\Big(\frac{C_0(n-1)}{C^{\alpha}_{\mathcal{H}}(E)}\Big)^{1/\alpha}(m_n-m_1)^{1/\alpha-1}\Bigg)^{n-1}\sup_{t\in E}|p(t)|.
		\end{equation}
		Moreover, the result is sharp in the following sense: one cannot replace the above constant by 
		$
		C_1((m_n-m_1)^{1/\alpha-1-\varepsilon})^{n-1}
		$
		with $\varepsilon>0$ and  a  constant $C_1$ depending  only on $E$ and $n$.
	\end{theorem}
	
	\begin{remark}\label{rmk1.1}
		$(i)$  To our best knowledge, Theorem  \ref{thm1} seems to be the first result which extends the classical  Tur\'{a}n--Nazarov inequality to the fractal setting. 
		By letting $\alpha\to1$ in \eqref{eq3.1}, we recover the classical Tur\'{a}n--Nazarov inequality \eqref{eq-Turan}, up to  some constant $C_0>0$.
		In contrast to \eqref{eq-Turan}, where the constant depends only on $|E|$ and the degree $n$,  our result reveals that in the fractal case the constant additionally depends on the frequency bandwidth $m_n-m_1$.
	
	$(ii)$    To demonstrate the failure of \eqref{equ3.1}, we consider the trigonometric polynomial 
	\begin{equation}\label{eq-counter01}
		p_{j}(t)=2i\sin(2\pi N_jt)=e^{2\pi itN_j}-e^{-2\pi itN_j},
	\end{equation}
	with a lacunary frequency $N_j=a^{b^j}$ for some $a, b>0$.  We then construct a set $E$ with $C^{\alpha}_{\mathcal{H}}(E)>0$ such that $\sup_{t\in E}|p_{j}(t)|\to 0$ as $j\to \infty$, thereby disproving \eqref{equ3.1}. 
	On the other hand, to prove \eqref{eq3.1}, a crucial step is to estimate the sublevel set concerning the logarithmic derivative of polynomials 
	$$
	E(P, y) := \left\{ z \in T : \left| \frac{d}{dz} \log P(z) \right| > y \right\}, \quad y > 0,
	$$
	a result that may be of independent interest, we refer to Lemma \ref{lemma-cartan} for detail.

\end{remark}



We will apply  the construction in Theorem \ref{thm1}  to derive negative  results for observability  and  unique continuation for Schr\"{o}dinger equations on fractal sets, see Theorems \ref{thm3} and \ref{thm4} below.

In order to better present our results and provide a comparison, we first consider the situation for the heat equations, which has been studied extensively. Consider 
\begin{align}\label{equ-heat-high-bd}
\left\{
\begin{array}{l}
	\partial_tu(t,x)-\Delta u(t,x)=0,\quad\, t>0,x\in \Omega, \\
	u(t,x)=0,\quad\quad \quad \quad \quad  t\geq 0,x\in \partial\Omega,\\
	u(0,\cdot)\in L^2(\Omega),
\end{array}
\right.
\end{align}
where   $\Omega\subset \mathbb{R}^d$ (with $d\geq 1$) is a bounded domain with a smooth boundary $\partial\Omega$.
The observability inequality states that for each \(T>0\), there exists a constant \(C_{\text{obs}} > 0\) such that
\begin{align}\label{equ-ob-intro}
\|u(T,\cdot)\|_{L^2(\Omega)}\leq C_{obs}\left(\int_0^T\int_E|u(t,x)|^2\d x\d t \right)^{1/2}.
\end{align}

When $E\subset\Omega$ is open,   the inequality   \eqref{equ-ob-intro} was derived, in general, by a global Carleman estimate
\cite{FurIma}, or the spectral inequality method \cite{Lebeau1995contr}.
When $E\subset\Omega$ is a subset of positive $d$-dim Lebesgue measure,
\eqref{equ-ob-intro} was proved in \cite{ApraizNullcontrol2013}.

Recently, Burq and Moyano  \cite{BM}  obtained a new observability inequality on a $W^{2,\infty}$ compact manifold $(\Omega, g)$ of dimension $d\ge 1$, allowing  $E\subset\Omega$  with zero Lebesgue measure.
Specifically, if $C^{d-1+\delta}_{\mathcal{H}}(E)>0$ for some $\delta\in(0,1)$ close to $1$,  then for each $T>0$, there is a constant $C_{obs}>0$,  such that any solution $u$ to \eqref{equ-heat-high-bd}, with the  operator $-\Delta$ replaced by 
a uniformly elliptic operator $-\mbox{div}(A(x)\nabla\cdot)$,
satisfies
\begin{equation}\label{eq-heat}
\|u(T,\cdot)\|_{L^2(\Omega)}\leq C_{obs}\int_0^T\sup_{x\in E}|u(t,x)| \d t.
\end{equation}
Further, when $\Omega\subset \mathbb{R}^d$ is bounded, open and connected, Green, etc. \cite{GLMO} proved that \eqref{eq-heat} holds for all sets $E$ satisfying $C^{d-1+\delta}_{\mathcal{H}}(E)>0$ with $\delta\in (0,1)$.  For related studies on both bounded domains and the whole space $\R^d$, we refer to \cite{HWW},  where the observation sets are measured by log-type Hausdorff contents, a condition that allows the inclusion of certain sets with Hausdorff dimension  $d-1$.

\begin{definition}
We say that $E\subset\Omega$ is an \emph{observable set} for the heat equation \eqref{equ-heat-high-bd} if inequality \eqref{equ-ob-intro} holds when $|E|>0$, or inequality \eqref{eq-heat} holds when $|E|=0$.
\end{definition}

Now we turn to the case of the Schr\"{o}dinger equations, where \emph{observable sets} can be defined analogously. Consider
\begin{equation}\label{equation1.2}
\begin{cases}
	i\partial_{t}u(t,x)+\Delta u(t,x)=0,\;t>0, x\in \Omega,\\
	u(0,x)=u_0\in L^2(\Omega).
\end{cases}
\end{equation}

When $\Omega$ is a general compact manifold, and $\Delta$ denotes the Laplace–Beltrami operator on $\Omega$, with Dirichlet or Neumann boundary condition when $\partial\Omega\ne 0$. 
It was shown in  \cite{Le92} (see also \cite{KDPhung})
that any open set with geometric control
condition (GCC) is \emph{an observable set}.
However, GCC is not  necessary on  the flat torus $\mathbb{T}^d := (\R/2\pi\mathbb{Z})^d$ (see, for instance,  \cite{Bour13,Ha,J90})  or on hyperbolic surfaces \cite{Jin} (see also \cite{Dya} for Anosov surfaces), where every non-empty open set is \emph{observable}.
In the context of rough localization,  it was verified in \cite{BZ17} that  on $\mathbb{T}^d$ for $d=1,2$, every subset $E$ of positive measure is \emph{an observable set}. Yet for $\Omega=\mathbb{T}^d$ with $d\ge3$, it remains unknown  whether   \eqref{equ-ob-intro} holds  for solutions of \eqref{equation1.2} whenever  $|E|>0$.

Compared to the heat equation, it is natural to ask whether the standard observability inequality 
\eqref{equ-ob-intro}  extends to the fractal setting   \eqref{eq-heat}. 
While the solution $e^{it\Delta}u_0$ remains in $L^2(\Omega)$ for any initial data $u_0\in L^2(\Omega)$, the quantity $\sup_{x \in E} |(e^{it\Delta}u_0)(t, x)|$ may be infinite when evaluated on a set $E$ of zero Lebesgue measure, regardless of the choice of $t\ge 0$. 
On the other hand, for initial data $u_0\in H^{k}(M)$ with $k>\frac{d}{2}$, the solution gains continuity through the standard Sobolev embedding theorem. Despite this regularity gain, we shall prove that for $\mathbb{T}$,  even for smooth initial data,
an observability inequality analogous to \eqref{eq-heat} fails to hold when the Hausdorff dimension of the observation set $E$ is strictly less than 1. 
More precisely,	we have

\begin{theorem}\label{thm3}
Let  $\Omega=\mathbb{T}$ in \eqref{equation1.2}. For every $\alpha\in (0,1)$, no observability inequality of the form  \eqref{eq-heat} holds for solutions  of \eqref{equation1.2} for  all subsets $E\subset\mathbb{T}$ with  $C^{\alpha}_{\mathcal{H}}(E)>0$, even if  $u_0(x)\in C^\infty(\mathbb{T})$.
\end{theorem}
Applying methods in Theorem \ref{thm1},
we can also establish negative results concerning unique continuation from fractal sets for  $\Omega=\R$ in \eqref{equation1.2}. To motivate the result, we recall
the following unique continuation inequality at two different time instances proved in  \cite{WWZ}: 
let $E_1=B_{r_1}(x_1)^c$, $E_2=B_{r_2}(x_2)^c$ be the complements of balls centered at $x_1, x_2\in \R^d$ with radii $r_1, r_2>0$ respectively. Then for any  $u_0\in L^2(\R^d)$, the solutions of the equation \eqref{equation1.2} satisfy
\begin{equation}\label{eq-two-point}
\int_{\mathbb{R}^d} |u_0(x)|^2\,dx \leq C e^{ \frac{Cr_1 r_2}{T-S}} \left(\int_{E_1} |u(S,x)|^2 \,dx + \int_{E_2} |u(T,x)|^2 \,dx\right),
\end{equation}
where $T>S\ge 0$ and $C>0$. 
To state our result, we require the following two types of sets, both of which describe a certain ``thick" property of a set.

\begin{definition}\label{def-set}
Let  $E \subset \mathbb{R}$.

$(i)$ For $\epsilon\in (0, 1]$, we say that $E$ is $\epsilon$ -thin (with respect to $\rho$) if
\begin{equation}\label{eq-def-rho-1024}
	|E \cap B(x, \rho(x))| \leq \epsilon \rho(x),\qquad \forall\, x\in \mathbb{R},
\end{equation}
where 
\begin{equation}\label{eq-rho}
	\rho(x) = \min\left(1, |x|^{-1}\right).
\end{equation}

$(ii)$  For $\alpha\in (0, 1]$, we   say that  $E$  is $\alpha$-thick if there are $L>0$ and $\gamma>0$ such that
\begin{equation}\label{thick-set}
	C^{\alpha}_{\mathcal{H}}(E_i \cap [x,x+L])>\gamma L,\qquad  \forall\, x\in \mathbb{R}.
\end{equation}
\end{definition}

The $\epsilon$ -thin set was introduced in \cite{SVW}, motivated by  problems arising in Anderson-Bernoulli models. The notion of $\alpha$-thick set was employed in \cite{HWW} to study observable sets for heat equations on fractal sets. When $\alpha=1$, it coincides with the usual thick sets, which is used to characterize observable sets for heat equation in $\R^d$, $d\ge 1$, see e.g. in \cite{wang2019}.

\begin{theorem}\label{thm4}
Let $\Omega=\R$ in \eqref{equation1.2} and  let  $0\le S<T$.  Assume that $u_0\in \bigcap_{k=0}^{\infty}H^{k}(\R)$.

$(i)$   There exist $\epsilon > 0$ and $C>0$,  such that if $E_1^c$ and $E_2^c$ are $\epsilon$-thin sets  with respect to $\rho$ given by \eqref{eq-rho}, then the solution $u(t,x)$ of \eqref{equation1.2} satisfies the unique continuation inequality
\begin{equation}\label{equ4.2}
	\|u_0\|_{L^2(\R)}^2\leq C\,\bigg(\big\|\big\{\sup_{E_1\cap[k,k+1]} |u(S,x)|\big\}_{k\in\mathbb{Z}}\big\|_{l^2(\mathbb{Z})}^2+\big\|\big\{\sup_{E_2\cap[k,k+1]} |u(T,x)|\big\}_{k\in\mathbb{Z}}\big\|_{l^2(\mathbb{Z})}^2\bigg).
\end{equation}

$(ii)$ For every $\alpha\in (0,1)$, there exist  $\alpha$-thick subsets $E_1, E_2\subset \mathbb{R}$ such that the inequality \eqref{equ4.2} does not hold.
\end{theorem}

%
%
\begin{remark}\label{rmk1.1}
	$(i)$ Unlike the corresponding results for the heat equation  \cite{BM,GLMO,HWW}, Theorem \ref{thm3} reveals a fundamental difference  in the observability for  Schr\"{o}dinger  equations  on fractal sets. This distinction arises from the following key observations: for Schr\"{o}dinger equations,  the failure of the observability inequality \eqref{eq-heat} stems from the same constructive approach  used to disprove \eqref{equ3.1} on fractal sets, where the constant also depends on the bandwidth $m_n-m_1$.
	In contrast, for heat equations, consider the spectral projector operator
	\begin{equation*}
		\Pi_{\Lambda} u = \sum_{\lambda_k \leq \Lambda} \langle u, \varphi_k \rangle_{L^2} \varphi_k, \quad \forall u \in L^2(\Omega).
	\end{equation*}
	Here, $\Omega\subset \R^d$ is a bounded Lipschitz domain and $\{\varphi_k\}$ are Laplace eigenfunctions associated to the
	eigenvalues $\{\lambda_k\}$.
	The following spectral inequality holds (see e.g. in \cite{BM,GLMO,HWW}):
	\begin{equation}\label{eq-spect-01}
		\| \Pi_{\Lambda} u \|_{L^2(\Omega)} \leq Ce^{C\sqrt{\Lambda}} \sup_{x \in E} | (\Pi_{\Lambda} u)(x)|, \quad \forall u \in L^2(\Omega),
	\end{equation}
	provided $E\subset\Omega $ satisfies $C_\mathcal{H}^{d-1+\delta}(E) > 0$. Notably, the constant $e^{C\sqrt{\Lambda}}$ remains identical to the case $|E|>0$. Combining  \eqref{eq-spect-01}  with the classical Lebeau-Robbiano strategy yields the observability estimate \eqref{eq-heat}.

	$(ii)$  For Schr\"{o}dinger equations on $\R$, 
	Theorem \ref{thm4} shows that \eqref{equ4.2} does not extend  to the fractal setting. To our best knowledge, a complete characterization of the sets $E_1, E_2$ for which \eqref{equ4.2} remains valid is still unknown. This question is closely related  to characterizing $E_1, E_2$ in  \eqref{eq-two-point}, and resolving it would likely constitute a significant result in the field.
	
	$(iii)$   The proofs of Theorem \ref{thm3} and \ref{thm4} rely heavily on the construction from part $(i)$ of Theorem \ref{thm1}, along with a careful choice of initial data. In the case of Theorem \ref{thm4}, we consider
	$$
	f_j(x)=\sin{(N_jx)}\left(e^{-(x+\pi N_j)^2/2}-e^{-(x-\pi N_j)^2/2}\right), \qquad j\in\mathbb{N},
	$$
	with parameters $N_j$ taken from \eqref{eq-counter01}. These functions satisfy  $\lim_{j\to\infty}\|f_j\|_{L^2(\R)}^2=\sqrt{\pi}$, and  both $f_j$ and their Fourier transforms $\hat{f}_j$ decay rapidly on the set $E$.

\end{remark}

The rest of the paper is organized as follows:
In Section \ref{section2}, we prove Theorem \ref{thm1}. 
Section \ref{section3}  presents the proofs of Theorems \ref{thm3} and \ref{thm4}.

\section{Proof of Theorem \ref{thm1}}\label{section2}

In subsection \ref{section-2-1}, for  any given  $\alpha\in (0, 1)$, we construct  a set $E\subset [0,1]$ with $C^{\alpha}_{\mathcal{H}}(E)>0$
for which the inequality \eqref{equ3.1} fails. 
In subsection \ref{section-2-2}, we prove  \eqref{eq3.1}.

\subsection{The proof of  statement $(i)$ }\label{section-2-1}

We divide the proof into three steps.

\emph{Step 1. The construction of $E$.}
Choose $\tilde{\alpha}\in (\alpha, 1)$ and $c>0$ such that
\begin{align}\label{equ2.1.7}
	0<c<\frac{1}{1+\tilde{\alpha}^{-\frac{1}{2}}},\qquad \frac{1}{c}\in\mathbb{N}.
\end{align}
Fix $0<\delta<\frac{1}{6}$ such that $\frac{1}{2\delta}\in \mathbb{N}$ and set
\begin{equation}\label{equ2.11}
	N_k=(\frac{1}{2\delta})^{s_k},\;\mbox{ where }s_k=(\frac{1}{c})^{k}.
\end{equation}
Clearly, we have $N_k\in\mathbb{N}$,  $N_k<N_{k+1}$ and $\lim_{k\to\infty}N_k=\infty$.
We define 
\begin{equation}\label{equ2.3}
	J_{m,\delta}(j)=[\frac{1}{2N_j}m,\frac{1}{2N_j}m+\frac{\delta}{N_j^{1/\beta}}]\cap [0,1],
\end{equation}
where 
\begin{align}\label{equ2.1.8}
	\beta=c+(1-c)\tilde{\alpha}.
\end{align}
The above definition of $\beta$ yields that $\tilde{\alpha}<\beta<1$.    Now we define
\begin{equation}\label{equ2.4}
	E= {\bigcap\limits_{j=1}^\infty} {\bigcup\limits_{m=0}^{2N_j-1}}J_{m,\delta}(j)=\bigcap\limits_{k=0}^\infty E_k,
\end{equation}
where 
\begin{equation}\label{equ2.5}
	E_0=[0,1],\,\,\,E_k={\bigcap\limits_{j=1}^k} {\bigcup\limits_{m=0}^{2N_j-1}}J_{m,\delta}(j),\,\,\,k=1,2,\cdots.
\end{equation}
The intervals obtained at each step are illustrated in the following figure:

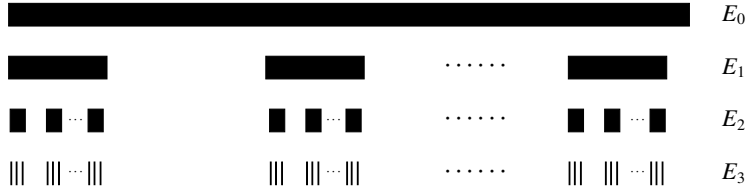
\begin{figure}[htbp]
	\centering
	\begin{tikzpicture}[scale=1]
		
		\def\width{9}
		\def\height{0.3}
		\def\gap{0.4}
		
		\filldraw[black] (0, 0) rectangle ++(\width, -\height);
		\node[anchor=west, font=\small] at (\width+0.3, -\height/2) {$E_0$};
		
		\pgfmathsetmacro{\yone}{-1*(\height + \gap)}
		\pgfmathsetmacro{\xA}{0.0}
		\pgfmathsetmacro{\xB}{3.4}
		\pgfmathsetmacro{\xC}{7.4}
		\pgfmathsetmacro{\wone}{1.3}
		\filldraw[black] (\xA, \yone) rectangle ++(\wone, -\height);
		\filldraw[black] (\xB, \yone) rectangle ++(\wone, -\height);
		\filldraw[black] (\xC, \yone) rectangle ++(\wone, -\height);
		\node at (6.2, \yone-0.15) {$\cdots\cdots$};
		\node[anchor=west, font=\small] at (\width+0.3, \yone-\height/2) {$E_1$};
		
		\pgfmathsetmacro{\ytwo}{-2*(\height + \gap)}
		\filldraw[black] (0.02, \ytwo) rectangle ++(0.2, -\height);
		\filldraw[black] (0.5, \ytwo) rectangle ++(0.2, -\height);
		\node[inner sep=0pt, scale=0.5] at (0.9, \ytwo-0.15) {$\cdots$};
		\filldraw[black] (1.05, \ytwo) rectangle ++(0.2, -\height);
		\filldraw[black] (3.45, \ytwo) rectangle ++(0.2, -\height);
		\filldraw[black] (3.93, \ytwo) rectangle ++(0.2, -\height);
		\node[inner sep=0pt, scale=0.5] at (4.31, \ytwo-0.15) {$\cdots$};
		\filldraw[black] (4.46, \ytwo) rectangle ++(0.2, -\height);
		\filldraw[black] (7.4, \ytwo) rectangle ++(0.2, -\height);
		\filldraw[black] (7.88, \ytwo) rectangle ++(0.2, -\height);
		\node[inner sep=0pt, scale=0.5] at (8.33, \ytwo-0.15) {$\cdots$};
		\filldraw[black] (8.48, \ytwo) rectangle ++(0.2, -\height);
		\node[anchor=west, font=\small] at (\width+0.3, \ytwo-\height/2) {$E_2$};
		
		\node at (6.2, \ytwo-0.15) {$\cdots\cdots$};
		
		\pgfmathsetmacro{\ythree}{-3*(\height + \gap)}
		\foreach \x in {0.03,0.1,0.18,0.51,0.58,0.66,1.06,1.13,1.21,3.46,3.53,3.61,3.94,4.01,4.09,4.47,4.53,4.61,7.41,7.48,7.56,7.89,7.96,8.04,8.49,8.56,8.64} {
			\filldraw[black] (\x, \ythree) rectangle ++(0.01, -\height);
		}
		\node at (6.2, \ythree-0.15) {$\cdots\cdots$};
		\node[anchor=west, font=\small] at (\width+0.3, \ythree-\height/2) {$E_3$};
		\node[inner sep=0pt, scale=0.5] at (0.9, \ythree-0.15) {$\cdots$};
		\node[inner sep=0pt, scale=0.5] at (4.31, \ythree-0.15) {$\cdots$};
		\node[inner sep=0pt, scale=0.5] at (8.33, \ythree-0.15) {$\cdots$};
		
	\end{tikzpicture}
	\label{fig:energy_level}
\end{figure}
Furthermore, suppose that each interval of $E_{k-1}$ contains at least $m_k$ intervals of $E_k\,(k=1,2,\cdots)$ which are separated by gaps of at least $\varepsilon_k$, where $\varepsilon_k>0$ for each $k$. Therefore, it follows from the definition of $E$ that
\begin{equation}\label{equ2.8}
	m_1=2N_1,\,\,\,m_k=\lfloor \frac{\frac{\delta}{N_{k-1}^{1/\beta}}}{\frac{1}{2N_k}} \rfloor=\lfloor 2\delta N_kN_{k-1}^{-1/\beta}\rfloor,\;\,\,\,k\geq 2,
\end{equation}
and 
\begin{equation}\label{equ2.9}
	\varepsilon_k=\frac{1}{2N_k}-\frac{\delta}{N_k^{1/\beta}},\;\,\,\,k\geq 1.
\end{equation}
Here, we denote by $\lfloor x\rfloor$ the greatest integer not exceeding $x$.
Before proceeding, we observe that $m_k$ and $\varepsilon_k$ satisfy 
\begin{equation}\label{eq-m-k-01}
	m_k\geq 2,\qquad k\geq 1, 
\end{equation}
and
\begin{equation}\label{eq-epsi-k-02}
	0<\varepsilon_{k+1}<\varepsilon_k,\qquad k\geq 1.
\end{equation}	

In fact, we have $m_1=2N_1\geq 2$. For $k\ge 2$, we have
\begin{align}\label{equ2.1}
	2\delta N_kN_{k-1}^{-1/\beta}=(\frac{1}{2\delta})^{(\frac{1}{c})^{k-1}(\frac{1}{c}-\frac{1}{\beta})-1}\ge(\frac{1}{2\delta})^{\frac{1}{c}(\frac{1}{c}-\frac{1}{\beta})-1}\ge \frac{1}{2\delta}\ge 3,\quad k\ge 2,
\end{align}
where the first equality follows from (\ref{equ2.11}), the second inequality uses the fact  $c<\frac{1}{2}$ (derived from  $\tilde{\alpha}<1$ and  \eqref{equ2.1.7}) and the fact $\frac{1}{c}-\frac{1}{\beta}>1$ (derived from  \eqref{equ2.1.7} and \eqref{equ2.1.8}). The last inequality follows from our choice $\delta<\frac{1}{6}$.  
Combining \eqref{equ2.1} with the fact $m_k\geq 2\delta N_kN_{k-1}^{-1/\beta}-1$( by (\ref{equ2.8})) yields \eqref{eq-m-k-01}.

To prove \eqref{eq-epsi-k-02},  
it suffices to show that $\varepsilon_{k+1} < \varepsilon_k$ for all $k$. Indeed, we have
\begin{align*}
	\varepsilon_k-\varepsilon_{k+1}&>\frac{1}{2N_k}-\frac{\delta}{N_k}-\frac{1}{2N_{k+1}}\\
	&\ge (2\delta)^{(\frac{1}{c})^{k}}(\frac{1}{2}-\delta-\frac{1}{2}(2\delta)^{\frac{1}{c}(\frac{1}{c}-1)})\\
	&>0,
\end{align*}
where  
the first inequality follows from \eqref{equ2.9} and the facts $\beta<1$ and $\delta>0$, while the third inequality follows from the bounds $\delta<\frac{1}{6}$ and $c<\frac{1}{2}$. 

\emph{Step 2. We prove that $C^{\alpha}_{\mathcal{H}}(E)>0$.}
Indeed, since $0 < \alpha < \tilde{\alpha}$, it is enough to prove $\dim_{\mathcal{H}} E = \tilde{\alpha}$.  
By \eqref{eq-m-k-01} and \eqref{eq-epsi-k-02}, we have the lower bound (see \cite[Example 4.6]{F})	
\begin{equation}\label{equ2.1.3}
	\dim_{\mathcal{H}}E\ge \liminf_{k\to\infty}\frac{\log(m_1\cdots m_{k-1})}{-\log(m_k\varepsilon_k)}.
\end{equation}

We first establish the following two facts.

\emph{Fact one:} Denote by $A_k:=m_1\cdots m_{k-1}$, then
\begin{equation}\label{equ-lowerbound}
	\liminf_{k\to\infty}\frac{\log A_k}{-\log(m_k\varepsilon_k)}=\liminf_{k\to\infty}\frac{\log A_k}{-\log(\delta N_{k-1}^{-1/\beta})}.
\end{equation}
Indeed, set $\varphi_k=\frac{m_k\varepsilon_k}{\delta N_{k-1}^{-1/\beta}}$, by \eqref{equ2.8}, \eqref{equ2.9}, we have
\begin{equation*}
	\varphi_k=\frac{\lfloor 2\delta N_kN_{k-1}^{-1/\beta}\rfloor}{2\delta N_kN_{k-1}^{-1/\beta}}-\frac{1}{N_k^{1/\beta}N_{k-1}^{-1/\beta}}.
\end{equation*}
By \eqref{equ2.1}, we have $N_kN_{k-1}^{-1/\beta}\to\infty$ as $k\to\infty$. Moreover, it follows from  \eqref{equ2.11} and the fact $c<\frac{1}{2}$ that $N_kN_{k-1}^{-1}=(\frac{1}{2\delta})^{(\frac{1}{c})^{k-1}(\frac{1}{c}-1)}\to\infty$. Thus, we have
$\lim_{k\to\infty}\varphi_k=1,$
which implies that
\begin{align*}
	\liminf_{k\to\infty}\frac{\log A_k}{-\log(m_k\varepsilon_k)}=\liminf_{k\to\infty}\frac{\log A_k}{-\log(\delta N_{k-1}^{-1/\beta}\varphi_k)}=\liminf_{k\to\infty}\frac{\log A_k}{-\log(\delta N_{k-1}^{-1/\beta})}.
\end{align*}
This proves  \eqref{equ-lowerbound}.

\emph{Fact two:} There exist positive constants $c_1$ and $c_2$, independent of $k$, such that
\begin{equation}\label{eq-fact-two}
	c_1C_k\le A_k\le c_2C_k,\qquad k\ge 1,  
\end{equation}
where 
\begin{equation}\label{equ2.1.4}
	C_k:=(\frac{1}{2\delta})^{(\frac{1}{c})^{k-1}-(\frac{1}{\beta}-1)\frac{(\frac{1}{c})^{k-2}-1}{1-c}-(k-2)}.
\end{equation}
Indeed, by \eqref{equ2.11} and \eqref{equ2.8}, we have
\begin{align*}
	A_k\leq 2N_1\cdot(2\delta N_2N_{1}^{-1/\beta})\cdots(2\delta N_{k-1}N_{k-2}^{-1/\beta})
	=2(\frac{1}{2\delta})^{(\frac{1}{c})^{k-1}-(\frac{1}{\beta}-1)\frac{(\frac{1}{c})^{k-2}-1}{1-c}-(k-2)}=2C_k.
\end{align*}
On the other hand, we have
\begin{align*}
	\frac{A_k}{C_k}\geq \frac{2N_1\cdot(2\delta N_2N_{1}^{-1/\beta}-1)\cdots(2\delta N_{k-1}N_{k-2}^{-1/\beta}-1)}{N_1\cdot(2\delta N_2N_{1}^{-1/\beta})\cdots(2\delta N_{k-1}N_{k-2}^{-1/\beta})}=2\prod_{j=1}^{k-2}(1-\frac{1}{2\delta N_{j+1}N_{j}^{-1/\beta}}).
\end{align*}
Denote by
\begin{equation}\label{equ2.12}
	a_j:=\frac{1}{2\delta N_{j+1}N_{j}^{-1/\beta}}=(\frac{1}{2\delta})^{-(\frac{1}{c})^{j}(\frac{1}{c}-\frac{1}{\beta})+1},\quad  j\in\mathbb{N}.
\end{equation}		
It follows from \eqref{equ2.1.7}, \eqref{equ2.1.8} and $0<\tilde{\alpha}<1$ that $\frac{1}{c}-\frac{1}{\beta}>1$ and $\frac{1}{c}>2$.
This implies 
\begin{equation*}
	a_j<(2\delta)^{2^j-1}<(\frac{1}{3})^{2^j-1}<1,
\end{equation*}
which, in turn, indicates that
$
\sum_{j=1}^{\infty}a_j<\infty.
$
Thus, there exists an absolute constant $c_1>0$ such that $A_k\ge c_1C_k$ holds. This proves \eqref{eq-fact-two}.

The above two facts yield
\begin{align*}
	\dim_{\mathcal{H}}E&\ge\liminf_{k\to\infty}\frac{\log(c_1 C_k)}{-\log(\delta N_{k-1}^{-1/\beta})}\\
	&=\liminf_{k\to\infty}\frac{\log c_1+\bigg({(\frac{1}{c})^{k-1}-(\frac{1}{\beta}-1)\frac{(\frac{1}{c})^{k-2}-1}{1-c}-(k-2)}\bigg)\log{(\frac{1}{2\delta})}}{\log 2+\bigg(\frac{1}{\beta}(\frac{1}{c})^{k-1}+1\bigg)\log(\frac{1}{2\delta})}\\
	&=\liminf_{k\to\infty}\frac{(\frac{1}{c})^{k-1}-(\frac{1}{\beta}-1)\frac{(\frac{1}{c})^{k-2}-1}{1-c}-(k-2)}{\frac{1}{\beta}(\frac{1}{c})^{k-1}}\\
	&=\beta\bigg(1-\frac{\frac{1}{\beta}-1}{1-c}\cdot c\bigg)
	=\tilde{\alpha},
\end{align*}
where the first inequality follows from \eqref{equ2.1.3}, \eqref{equ-lowerbound} and \eqref{eq-fact-two}; the first equality is derived from \eqref{equ2.11} and \eqref{equ2.1.4}; while the last equality follows directly from \eqref{equ2.1.8}.

On the other hand, we also have (see \cite[Proposition 4.1]{F})
\begin{align*}
	\dim_{\mathcal{H}}E\le \varliminf_{k\to\infty} \frac{\log (m_1\cdots m_k)}{-\log (\delta N_k^{-1/\beta})}.
\end{align*}
Similarly, we obtain
\begin{align*}
	\dim_{\mathcal{H}}E&\le \varliminf_{k\to\infty} \frac{\log \bigg((2\delta)^{k-1}(N_1\cdots N_{k-1})^{1-\frac{1}{\beta}}N_k\bigg)}{\frac{1}{\beta}\log N_k-\log \delta}\\
	&=\varliminf_{k\to\infty}\frac{\bigg((\frac{1}{c})^k-(\frac{1}{\beta}-1)\frac{(\frac{1}{c})^{k-1}-1}{1-c}-(k-1)\bigg)\log (\frac{1}{2\delta})}{\frac{1}{\beta}(\frac{1}{c})^k\log (\frac{1}{2\delta})-\log \delta}\\
	&=\beta\bigg(1-\frac{\frac{1}{\beta}-1}{1-c}\cdot c\bigg)
	=\tilde{\alpha}.
\end{align*}
Therefore, $\dim_{\mathcal{H}} E = \tilde{\alpha}$.

\emph{Step 3. We prove that \eqref{equ3.1} fails to hold uniformly on $E$.}
Consider the following  trigonometric polynomial of degree $2$:
\begin{equation}\label{equ3.12}
	p_{j}(t)=e^{2\pi itN_j}-e^{-2\pi itN_j}=2i\sin(2\pi N_jt),\quad t\in[0,1],
\end{equation}
where $N_j$ is given by \eqref{equ2.11}.
Clearly, we have
\begin{equation}\label{equ3.2}
	\sup_{t\in[0,1]}|p_{j}(t)|=2\sup_{t\in[0,1]}|\sin{(2\pi N_jt)}|=2.
\end{equation}
We claim that
\begin{align}\label{eq-limit}
	\sup_{t \in E} |\sin(2\pi N_j t)| \longrightarrow 0, \quad\mbox{ as }\, j \to \infty.
\end{align}
Since $t\in E$, then for any $j\in\mathbb{N}$, there exists some $m\in\{0,1,\cdots,2N_j-1\}$, such that
$t\in [\frac{1}{2N_j}m,\frac{1}{2N_j}m+\frac{\delta}{N_j^{1/\beta}}].$
Therefore,
$$0\le2\pi N_jt-m\pi\le\frac{2\pi\delta}{N_j^{1/\beta-1}}<\frac{\pi}{2}.$$
Observe that	$0<\beta<1$ (by \eqref{equ2.1.8}) and  $\lim_{j\to\infty}{N_j}=\infty$, we derive that
\begin{align}\label{equ2.7}
	|\sin(2\pi N_jt)|<|\sin\frac{2\pi\delta}{N_j^{1/\beta-1}}|\leq \frac{2\pi\delta}{N_j^{1/\beta-1}}\to 0, \quad\mbox{ as }\, j \to \infty.
\end{align}
This implies \eqref{eq-limit} and the proof of statement $(i)$ is complete.

\subsection{The proof of  statement $(ii)$.}\label{section-2-2}

The key to establishing inequality \eqref{eq3.1} relies on an estimate  for the logarithmic derivative of polynomials. More precisely, we have
\begin{lemma}\label{lemma-cartan}
	Let $\alpha\in(0,1)$ and $H>0$, there exists $C_0>0$ such that for every algebraic polynomial $P(z)$ of degree $m\ge 1$, the following inequality hold: 
	\begin{equation}\label{equ-car-T}
		C^{\alpha}_{\mathcal{H}}\bigg(\Big\{z\in\mathbb{T}:\;|\frac{d}{dz}\log P(z)|>H\Big\}\bigg)\leq C_0mH^{-\alpha}.
	\end{equation}
\end{lemma}
\begin{proof}
	Assuming that $z_1,\cdots,z_m$ are the zeros of $P(z)$, then 
	\begin{equation*}
		C^{\alpha}_{\mathcal{H}}\bigg(\Big\{z\in\mathbb{T}:\;|\frac{d}{dz}\log P(z)|>H\Big\}\bigg)=C^{\alpha}_{\mathcal{H}}\bigg(\Big\{z\in\mathbb{T}:\;|\sum_{j=1}^m\frac{1}{z-z_j}|>H\Big\}\bigg).
	\end{equation*}
	When $m=1$, it is obviously that
	\begin{equation*}
		C^{\alpha}_{\mathcal{H}}\bigg(\Big\{z\in\mathbb{T}:\;|\frac{1}{z-z_1}|>H\Big\}\bigg)\le H^{-\alpha}.
	\end{equation*}
	When $m\ge 2$, we adapt some ideas  from \cite[Theorem 2.4]{ENV}. If \eqref{equ-car-T} is not true, we assume
	\begin{equation}\label{equ-M-define}
		\textbf{M}:=C^{\alpha}_{\mathcal{H}}\bigg(\Big\{z\in\mathbb{T}:\;|\sum_{j=1}^m\frac{1}{z-z_j}|>H\Big\}\bigg)>C_1 mH^{-\alpha}
	\end{equation}
	with sufficiently large $C_1>0$.
	The proof relies on the following claim:
	
	\emph{Claim: There exists $F\subset\mathbb{T}$ satisfying
		\begin{equation}\label{equ-F-lowerbound}
			C^{\alpha}_{\mathcal{H}}(F)\ge 0.6\textbf{M},
		\end{equation}
		and
		\begin{equation}\label{equ-subset}
			B(z_0,\frac{0.4}{H})\bigcap\mathbb{T}\subset \Big\{z\in\mathbb{T}:\;|\sum_{j=1}^m\frac{1}{z-z_j}|>0.8H\Big\}\quad\mbox{for every }z_0\in F.
		\end{equation}
	}
	Having introduced the \emph{Claim} (whose proof we postpone), let $\mu$ be the 1-dimensional Lebesgue measure and set
	\begin{equation*}
		G:=\bigcup_{z\in F}\big(B(z,\frac{0.4}{H})\cap\mathbb{T}\big).
	\end{equation*}
	By \eqref{equ-subset}, we have $G\subset\Big\{z\in\mathbb{T}:\;|\sum_{j=1}^m\frac{1}{z-z_j}|>0.8H\Big\}$. Applying  \cite[Lemma 1.2]{N}, we obtain
	\begin{equation}\label{eq-upper-bound}
		\mu(G)\le\mu\{z\in\mathbb{T}: |\sum_{j=1}^m\frac{1}{z-z_j}|>0.8H\}\le \frac{10}{\pi}mH^{-1}.
	\end{equation}
	On the other hand, the set $G$ can be covered by sets $B_j\cap\mathbb{T}$  of radius $\frac{1}{H}$  such  that
	\begin{equation*}
		\mu(G)\ge c\sum_j \frac{1}{H}.
	\end{equation*}
	Then, by \eqref{equ-F-lowerbound}
	\begin{align*}
		\mu(G)\ge c\sum_j \frac{1}{H}=cH^{\alpha-1}\sum_j (\frac{1}{H})^\alpha\ge cH^{\alpha-1}C_\mathcal{H}^\alpha(G)\ge cH^{\alpha-1}C_\mathcal{H}^\alpha(F)\ge 0.6cH^{\alpha-1}\textbf{M}.
	\end{align*}
	This, together with \eqref{eq-upper-bound}, implies
	\begin{equation*}
		\textbf{M}\le CmH^{-\alpha},
	\end{equation*}
	where $C=\frac{50}{3\pi c}$ is independent of $m$ and $H$. This completes the proof of Lemma \ref{lemma-cartan}. 
	
	It remains to prove the \emph{Claim}. We say that a point $z_0\in\mathbb{T}$ is \emph{normal} if the inequality
	\begin{equation*}
		\#\{z_j:z_j\in B(z_0,r)\}\le C_2^{-1}H^\alpha r^\alpha
	\end{equation*} 
	holds for all $r\ge 0$, where $C_2<C_1$ will be
	specified later. Let $G_1$ be the set of all \emph{non-normal} points $z\in\mathbb{T}$. Thus for each $z\in G_1$, there exists $r=r(z)$ such that 
	\begin{equation}\label{equ-r-bound}
		r^\alpha<C_2H^{-\alpha}\#\{z_j:z_j\in B(z_0,r)\}.
	\end{equation}
	This yields a covering of $G_1$ by sets $B(z,r(z))\cap\mathbb{T}$. 
	By the Besicovitch covering lemma, there is a subcover $\{B_k'\cap\mathbb{T}\}$, with $B_k'=B(w_k',r_k')$, of multiplicity at most  $A_0$ (i.e., every point $z\in G_1$ is covered by at most $A_0$ such sets). Let $\mathcal{Z}_1=\cup_k(B_k'\cap \mathbb{T})$, then from \eqref{equ-r-bound}, we obtain
	\begin{equation*}
		C^{\alpha}_{\mathcal{H}}(\mathcal{Z}_1)\le \sum_k(r_k')^\alpha<C_2H^{-\alpha}\sum_k\#\{z_j:z_j\in B(w_k',r_k')\}\le A_0C_2H^{-\alpha}m.
	\end{equation*}
	Combining this with \eqref{equ-M-define}, we derive that
	\begin{equation*}
		C^{\alpha}_{\mathcal{H}}(\mathcal{Z}_1)<0.3\textbf{M},\qquad \mbox{if}\,\, A_0C_2<0.3C_1.
	\end{equation*}
	Setting $\mathcal{Z}_2=\cup_{j=1}^m(B(z_j,\frac{1}{H})\cap\mathbb{T})$, we deduce from \eqref{equ-M-define} that
	\begin{equation*}
		C^{\alpha}_{\mathcal{H}}(\mathcal{Z}_2)\le mH^{-\alpha}<0.1\textbf{M},\qquad \mbox{if}\,\, C_1>10.
	\end{equation*}
	We define the set $F$ by
	\begin{equation}
		F=\Big\{z\in\mathbb{T}:\;|\sum_{j=1}^m\frac{1}{z-z_j}|>H\Big\}\setminus(\mathcal{Z}_1\cup\mathcal{Z}_2).
	\end{equation}
	It follows immediately that
	\begin{equation*}
		C^{\alpha}_{\mathcal{H}}(F)\ge \textbf{M}-0.3\textbf{M}-0.1\textbf{M}=0.6\textbf{M},
	\end{equation*}
	and  hence \eqref{equ-F-lowerbound} is satisfied.
	
	We next verify \eqref{equ-subset}.  
	Fix $z_0\in F$, for any $z\in B(z_0,\frac{0.4}{H})\cap\mathbb{T}$, we have 
	\begin{equation}\label{equ-F-prop1}
		|\sum_{j=1}^m\frac{1}{z_0-z_j}|>H,
	\end{equation}
	\begin{equation}\label{equ-F-prop2}
		|z_0-z_j|\ge \frac{1}{H},\quad\forall j,
	\end{equation}
	and
	\begin{equation}\label{equ-F-prop3}
		\#\{z_j:z_j\in B(z_0,r)\}\le C_2^{-1}H^\alpha r^\alpha,\,\;\forall r\ge 0. 
	\end{equation}
	Now define, for $k=0,1,\cdots,$
	\begin{equation*}
		B_k=\{j:\frac{2^k}{H}\le |z_0-z_j|<\frac{2^{k+1}}{H}\}.
	\end{equation*}
	By \eqref{equ-F-prop2}, we have $j\in\cup_{k=0}^\infty B_k$.
	Moreover, using \eqref{equ-F-prop2} and the fact $|z-z_0|\le \frac{0.4}{H}$, we obtain
	\begin{equation*}
		|z-z_j|\ge|z_0-z_j|-|z-z_0|>\frac{2^{k-1}}{H},\,\,\;\forall j\in B_k.
	\end{equation*}
	Thus,
	\begin{align*}
		&\quad\big|\sum_{j=1}^m(\frac{1}{z-z_j}-\frac{1}{z_0-z_j})\big|=\big|\sum_{j=1}^m\frac{z-z_0}{(z-z_j)(z_0-z_j)}\big|\\
		&\le \frac{0.4}{H}\sum_{k=0}^\infty\sum_{j\in B_k}\frac{H^2}{2^{2k-1}}\le 0.4H\sum_{k=0}^\infty\big(\frac{1}{2^{2k-1}}\cdot\#B_k\big)\\
		&\le 0.4H\sum_{k=0}^\infty\frac{1}{2^{2k-1}}\cdot C_2^{-1}H^\alpha(\frac{2^{k+1}}{H})^\alpha\quad\mbox{(by \eqref{equ-F-prop3})}\\
		&\le \frac{4H}{C_2}<0.2H,
	\end{align*}
	provided  $C_2>20$. 
	Hence, from \eqref{equ-F-prop1}, it follows that
	\begin{equation*}
		|\sum_{j=1}^m\frac{1}{z-z_j}|\ge  |\sum_{j=1}^m\frac{1}{z_0-z_j}|- |\sum_{j=1}^m(\frac{1}{z-z_j}-\frac{1}{z_0-z_j})|>H-0.2H=0.8H.
	\end{equation*}
	Since $z_0\in F$ is arbitrary, this establishes \eqref{equ-subset} and completes the proof of the Lemma.
\end{proof}

\begin{remark}
	We briefly recall some background on the estimate \eqref{equ-car-T} in Lemma \ref{lemma-cartan}. 
	In the case $\alpha=1$, Nazarov showed that \eqref{equ-car-T} holds with the explicit constant $C_0=\pi/8$. Moreover, when $\mathbb{T}$ is replaced by $\R$, he obtained an analogous inequality with $C_0=8$ (see \cite[lemma 1.2]{N}).  If $\mathbb{T}$ is replaced by $\mathbb{C}$, then, up to the value of the constant $C$, the optimal estimate is given by \cite[Theorem 3.2]{Eid} (see also \cite{AE1,AE2}): Let $P(z)$ be a polynomial of degree $m$, then
	\begin{equation*}
		C^{1}_{\mathcal{H}}\bigg(\Big\{z\in\mathbb{C}:\;|\frac{d}{dz}\log P(z)|>H\Big\}\bigg)\le C m(\ln m)^{1/2}H^{-1},\qquad H>0.
	\end{equation*}
	
	The case $\alpha\in(0,1)$ is more subtle. It was established in \cite[(4.2)]{Eid} (with $d=\alpha$, $N=\|\nu\|=m$, and $P=H$) that
	\begin{equation*}
		C^{\alpha}_{\mathcal{H}}\bigg(\Big\{z\in\mathbb{C}:\;|\frac{d}{dz}\log P(x)|>H\Big\}\bigg)\leq Cm(\frac{1}{1-\alpha})^{\alpha/2}H^{-\alpha}.
	\end{equation*}
	Moreover, the exponent $\alpha/2$ is sharp, see \cite[Theorem 4.4]{Eid}. However, if we replace $\mathbb{C}$ by $\R$, then  \cite[Theorem 2.4]{ENV} (with $d=s=1$, $N=\|\nu\|=m$, $P=H$ and $h(t)=t^\alpha$) implies
	\begin{equation}\label{equ-car-R}
		C^{\alpha}_{\mathcal{H}}\bigg(\Big\{z\in\mathbb{R}:\;|\frac{d}{dz}\log P(z)|>H\Big\}\bigg)\leq CmH^{-\alpha}.
	\end{equation}
\end{remark}

\emph{ The proof of statement $(ii)$.}
We express the exponential sum \eqref{eq-tri-01} in polynomial form by setting $z=e^{2\pi it}$, yielding 
$$
p=\sum_{k=1}^nc_ke^{2\pi im_kt}=\sum_{k=1}^n c_k z^{m_k}.
$$  
Furthermore, we introduce the norm
\begin{equation}\label{eq-norm}
	\|p\|_{W}:= \sum_{k=1}^n|c_k|.
\end{equation}
Thus,  to establish \eqref{eq3.1}, it suffices to prove that
\begin{align}\label{equ3.1.0}
	\|p\|_{W}\leq \Bigg\{\Big(\frac{4C_0(n-1)}{C^{\alpha}_{\mathcal{H}}(E)}\Big)^{1/\alpha}(m_n-m_1)^{1/\alpha-1}\Bigg\}^{n-1}\sup_{z\in E}|p(z)|,
\end{align}
where $E\subset\mathbb{T}$ satisfies $C^{\alpha}_{\mathcal{H}}(E)>0$. 
Following the approach in \cite{N}, the proof of \eqref{equ3.1.0} will be established through four steps.

\emph{Step 1: We construct a sequence of polynomials $p_n,\,p_{n-1},\,\cdots,\,p_1$ satisfying:}

1) Initial condition: $p_n=p$;

2)   Degree constraint: $\mbox{deg}\, p_k=k$, $\;\,\,1\le k\le n$\footnote{Throughout the proof, $\mbox{deg}$ denotes the degree of a trigonometric polynomial.};

3) Norm inequality: $\| p_{k-1}\|_W\geq \frac{1}{2}\| p_k\|_W$.

We begin with $p_n=p$.  For a given polynomial $p_k(z)=\sum_{s=1}^kd_sz^{r_s}$ (with strictly increasing integer exponents $r_1<r_2<\cdots<r_k)$,
We denote by  $r^{(k)}:=r_k-r_1$, which represents the difference between the highest and lowest degree of $p_k(z)$. 
Now we introduce two auxiliary polynomials
\begin{equation*}
	\underline{q}(z):= \frac{d}{dz}(z^{-r_1}p_k(z))\quad\mbox{ and }\quad \overline{q}(z):=  \frac{d}{dz}(z^{-r_k}p_k(z)).
\end{equation*}
Obviously, $\mbox{deg}\,\underline{q}=\mbox{deg}\,\overline{q}=k-1$. Moreover,
\begin{equation*}
	\| \underline{q}\|_W=\sum_{s=1}^k|d_s|(r_s-r_1),\quad\quad\| \overline{q}\|_W=\sum_{s=1}^k|d_s|(r_k-r_s),
\end{equation*}
whence
\begin{equation*}
	\| \underline{q}\|_W+\| \overline{q}\|_W=(r_k-r_1)\sum_{s=1}^k|d_s|=r^{(k)}\| p_k\|_W,
\end{equation*}
This implies at least one of the following must hold: 
$$\| \overline{q}\|_W\geq \frac{r^{(k)}}{2}\| p_k\|_W\quad\mbox{or} \quad \| \underline{q}\|_W\geq \frac{r^{(k)}}{2}\| p_k\|_W.$$
Without loss of generality, we consider the first case (the second case follows similarly). Define
\begin{equation}\label{eq-p-k-1-530}
	p_{k-1}(z)=\frac{1}{r^{(k)}}\overline{q}(z), 
\end{equation}
we immediately obtain the required norm inequality
\begin{equation*}
	\| p_{k-1}\|_W=\frac{1}{r^{(k)}}\| \overline{q}\|_W\geq\frac{1}{2}\| p_k\|_W.
\end{equation*}
This construction satisfies all three conditions 1), 2) and 3). 
Furthermore, the sequence satisfies the monotonicity property
\begin{align}\label{equ3.14}
	r^{(k-1)}\leq r^{(k)}\leq r^{(n)}=m_n-m_1,\quad \mbox{for}\,\,2\leq k\leq n.
\end{align}
To see this, observe from \eqref{eq-p-k-1-530} that
\begin{equation*}
	p_{k-1}(z)=\frac{1}{r^{(k)}}\sum_{s=1}^{k-1}d_s(r_s-r_k)z^{r_s-r_k-1},
\end{equation*}
which implies
\begin{equation*}
	r^{(k-1)}=r_{k-1}-r_1<r_k-r_1=r^{(k)}.
\end{equation*}

\emph{Step 2: We prove that there exists some $z_0\in E$, such that
	\begin{equation}\label{equ3.11}
		\sum_{k=2}^{n}\log\varphi_k(z_0)\leq \frac{n-1}{\alpha}\log\left(\frac{2C_0(n-1)(m_n-m_1)^{1-\alpha}}{C^{\alpha}_{\mathcal{H}}(E)}\right),
\end{equation}}
where   $\varphi_k:=|\frac{p_{k-1}}{p_k}|$, $2\le k\le n$.

We begin by observing the  polynomial representation:  
\begin{equation}\label{eq-p-k-g}
	z^{-r_k}p_k(z)=g(1/z),
\end{equation}
where $g$ is an algebraic polynomial of degree $r^{(k)}$. This yields
\begin{equation}\label{eq-q-g'}
	\overline{q}(z)=-\frac{1}{z^2}g'(1/z).
\end{equation}
Proceeding by contradiction, we assume that \eqref{equ3.11} does not hold: for every $z\in E$,
\begin{equation*}
	\sum_{k=2}^{n}\log\varphi_k(z)> \frac{n-1}{\alpha}\log\left(\frac{2C_0(n-1)(m_n-m_1)^{1-\alpha}}{C^{\alpha}_{\mathcal{H}}(E)}\right).
\end{equation*}
This assumption implies
\begin{align}\label{equ3.16}
	E\subset\bigg\{z\in\mathbb{T}:\;\sum_{k=2}^{n}\log\varphi_k(z)> \frac{n-1}{\alpha}\log\bigg(\frac{2C_0(n-1)(m_n-m_1)^{1-\alpha}}{C^{\alpha}_{\mathcal{H}}(E)}\bigg)\bigg\},
\end{align}
from which we derive
\begin{align}\label{equ3.17}
	C^{\alpha}_{\mathcal{H}}(E)
	&\leq \sum_{k=2}^{n}C^{\alpha}_{\mathcal{H}}\Bigg(\bigg\{z\in\mathbb{T}:\;\log\varphi_k(z)> \frac{1}{\alpha}\log\bigg(\frac{2C_0(n-1)(m_n-m_1)^{1-\alpha}}{C^{\alpha}_{\mathcal{H}}(E)}\bigg)\bigg\}\Bigg)\notag\\
	&=\sum_{k=2}^{n}C^{\alpha}_{\mathcal{H}}\Bigg(\bigg\{z\in\mathbb{T}:\;|\frac{p_{k-1}(z)}{p_k(z)}|> \bigg(\frac{2C_0(n-1)(m_n-m_1)^{1-\alpha}}{C^{\alpha}_{\mathcal{H}}(E)}\bigg)^{\frac{1}{\alpha}}\bigg\}\Bigg)\quad(\mbox{by }\varphi_k=|\frac{p_{k-1}}{p_k}|)\notag\\
	&=\sum_{k=2}^{n}C^{\alpha}_{\mathcal{H}}\Bigg(\bigg\{z\in\mathbb{T}:\;|\frac{g'(1/z)}{g(1/z)}|>r^{(k)} \bigg(\frac{2C_0(n-1)(m_n-m_1)^{1-\alpha}}{C^{\alpha}_{\mathcal{H}}(E)}\bigg)^{\frac{1}{\alpha}}\bigg\}\Bigg)\notag\\
	&=\sum_{k=2}^{n}C^{\alpha}_{\mathcal{H}}\Bigg(\bigg\{z\in\mathbb{T}:\;|\frac{d}{dz}\log g(z)|>r^{(k)} \bigg(\frac{2C_0(n-1)(m_n-m_1)^{1-\alpha}}{C^{\alpha}_{\mathcal{H}}(E)}\bigg)^{\frac{1}{\alpha}}\bigg\}\Bigg)\notag\\
	&\leq \sum_{k=2}^{n}C_0 r^{(k)} \Bigg(r^{(k)} \bigg(\frac{2C_0(n-1)(m_n-m_1)^{1-\alpha}}{C^{\alpha}_{\mathcal{H}}(E)}\bigg)^{\frac{1}{\alpha}}\Bigg)^{-\alpha}\quad(\mbox{by \eqref{equ-car-T} in Lemma \ref{lemma-cartan}})\notag\\
	&=\sum_{k=2}^{n} (r^{(k)})^{1-\alpha}\frac{C^{\alpha}_{\mathcal{H}}(E)}{2(n-1)(m_n-m_1)^{1-\alpha}}\notag\\
	&\leq \sum_{k=2}^{n} \frac{C^{\alpha}_{\mathcal{H}}(E)}{2(n-1)}\quad\quad(\mbox{by }\eqref{equ3.14}\mbox{ and }\alpha<1)\notag\\
	&=\frac{C^{\alpha}_{\mathcal{H}}(E)}{2}<C^{\alpha}_{\mathcal{H}}(E),
\end{align}
which leads to a contradiction. In \eqref{equ3.17}, the first inequality follows from \eqref{equ3.16} and subadditivity, while the second equality is justified by the identities \eqref{eq-p-k-1-530}, \eqref{eq-p-k-g} and \eqref{eq-q-g'}.
Consequently, \eqref{equ3.11} holds.

\emph{Step 3: We establish  \eqref{eq3.1}.}

In fact, \eqref{equ3.1.0} follows from the following sequence of inequalities:
\begin{align*}
	\|p\|_W&\leq 2^{n-1}\|p_1\|_W\stackrel{(deg\; p_1=1)}{=}2^{n-1}|p_1(z_0)|
	=2^{n-1}\prod_{k=2}^n\varphi_k(z_0)\cdot|p_n(z_0)|\\
	&\leq2^{n-1}\exp\bigg(\sum_{k=2}^n\log\varphi_k(z_0)\bigg)\cdot\sup_{z\in E}|p_n(z)|\quad(\mbox{by }z_0\in E)\\
	&\leq 2^{n-1}\Bigg(\frac{2C_0(n-1)(m_n-m_1)^{1-\alpha}}{C^{\alpha}_{\mathcal{H}}(E)}\Bigg)^{\frac{n-1}{\alpha}}\cdot\sup_{z\in E}|p_n(z)|\quad\;(\mbox{by \eqref{equ3.11}})\\
	&\le\Bigg(\Big(\frac{4C_0(n-1)}{C^{\alpha}_{\mathcal{H}}(E)}\Big)^{1/\alpha}
	(m_n-m_1)^{1/\alpha-1}\Bigg)^{n-1}\cdot\sup_{z\in E}|p(z)|.
\end{align*}

\emph{Step 4: We prove the sharpness of  \eqref{eq3.1}.}
Let $\varepsilon_0>0$ be arbitrarily small and take $p=p_{j}$ as defined in \eqref{equ3.12}. Let $\beta$ and $E$ be given by \eqref{equ2.1.8} and \eqref{equ2.4}, respectively. Recall that  $\dim_{\mathcal{H}}E=\tilde{\alpha}$, which implies $C^{\tilde{\alpha}-\varepsilon}_{\mathcal{H}}(E)>0$ for all $0<\varepsilon<\tilde{\alpha}$. By \eqref{equ2.1.7}, we may assume $c=\frac{1}{M}$ where $M\ge 3$ is an integer. Further, we suppose that
\begin{align*}
	M>\frac{\varepsilon_0\tilde{\alpha}^2+2(1-\tilde{\alpha})-\varepsilon_0\tilde{\alpha}}{\varepsilon_0\tilde{\alpha}^2},\quad\varepsilon\le\tilde{\varepsilon}:=\frac{\varepsilon_0\tilde{\alpha}\beta-\frac{2}{M}(1-\tilde{\alpha})}{2+\varepsilon_0\beta}.
\end{align*}
Under these conditions, it follows that $\tilde{\varepsilon}>0$ and
\begin{align}\label{equ3.15}
	\frac{1}{\tilde{\alpha}-\varepsilon}-\frac{1}{\beta}\le\frac{\varepsilon_0}{2}.
\end{align}
The selection of $p_j$ gives $n=2$ and $m_2-m_1=2N_j$. Combining \eqref{equ2.7} with \eqref{equ3.15}, we obtain
\begin{align*}
	&\quad\bigg\{(m_n-m_1)^{1/(\tilde{\alpha}-\varepsilon)-1-\varepsilon_0}\bigg\}^{n-1}\sup_{t\in E}|p(t)|= (2N_j)^{1/(\tilde{\alpha}-\varepsilon)-1-\varepsilon_0}\cdot\sup_{t\in E}|2\sin(2\pi N_jt)|\\
	&\le (2N_j)^{1/(\tilde{\alpha}-\varepsilon)-1-\varepsilon_0}\cdot2\cdot\frac{2\pi\delta}{N_j^{1/\beta-1}}=\frac{c}{N_j^{1/\beta-1/(\tilde{\alpha}-\varepsilon)+\varepsilon_0}}\le \frac{c}{N_j^{\varepsilon_0/2}}\to 0,\quad j\to\infty.
\end{align*}
Therefore, there exists $E\subset[0,1]$ with $C^{\alpha}_{\mathcal{H}}(E)>0$ (where $\alpha=\tilde{\alpha}-\varepsilon$) such that 
\begin{equation*}
	\sup_{t\in [0,1]}|p(t)|\le C\cdot\bigg((m_n-m_1)^{1/\alpha-1-\varepsilon_0}\bigg)^{n-1}\sup_{t\in E}|p(t)|
\end{equation*}
fails for $p=p_j$. Therefore we complete the proof of Theorem  \ref{thm1}.
\qed
\section{Proofs of Theorem \ref{thm3} and Theorem \ref{thm4}}\label{section3}

\subsection{The proof of Theorem \ref{thm3}.}

We begin by constructing a set $E\subset (0,\pi)$ with $C^{\alpha}_{\mathcal{H}}(E)>0$. Following the methodology of Theorem \ref{thm1}, we fix $\delta\in(0,1)$ such that $0<\delta<\frac{\pi}{3}$,  assume $\frac{\pi}{\delta}\in \mathbb{N}$ and take $\tilde{\alpha}\in (\alpha,1)$. 
Let $c$ be as defined in \eqref{equ2.1.7}, with $\frac{1}{c}\in\mathbb{N}$. We then define the sequence
\begin{equation}\label{equ4.3}
	N_j=(\frac{\pi}{\delta})^{s_j},\;\mbox{ where }s_j=(\frac{1}{c})^{j}.
\end{equation}
Thus $N_j\in\mathbb{N}$, $N_j<N_{j+1}$ and $\lim_{j\to\infty}{N_j}=\infty$. Let 
\begin{align}\label{equ4.4}
	J_{m,\delta}(j)=(\frac{\pi}{N_j}m,\frac{\pi}{N_j}m+\frac{\delta}{N_j^{1/\beta}})\cap (0,\pi),\qquad m\in \mathbb{Z},
\end{align}
where $\beta$ is given by \eqref{equ2.1.8}. 
The set $E$ is constructed through
\begin{equation}
	E= {\bigcap\limits_{j=1}^\infty} {\bigcup\limits_{m\in\mathbb{Z}}}J_{m,\delta}(j).
\end{equation}
Following the proof of Theorem \ref{thm1}, we have $\dim_{\mathcal{H}}E=\tilde{\alpha}$, which implies that $C^{\alpha}_{\mathcal{H}}(E)>0$.

We now prove that the observability inequality   \eqref{eq-heat} fails  for a class of $C^{\infty}$ initial data. Since the torus $\mathbb{T}$ can be identified with the interval $[0,2\pi)$ endowed with periodic boundary conditions. We consider the following initial data: 
\begin{align}\label{equ4.5}
	u_{0,j}=\sin{(N_jx)},\quad j\in\mathbb{N},
\end{align}
where $N_j$ is given by \eqref{equ4.3}. Obviously, $u_{0,j}$ is $2\pi$-periodic and thus naturally defined on $\mathbb{T}$. Moreover, $ u_{0,j}\in C^\infty(\mathbb{T})$ for any $j\in \mathbb{N}$, and 
\begin{equation}\label{eq-u_n-0-6-21}
	\| u_{0,j}\|_{L^2(\mathbb{T})}:=\| u_{0,j}\|_{L^2((0,2\pi))}=\sqrt{\pi},\qquad \mbox{for all}\,\,j\in\mathbb{N}.
\end{equation}
Furthermore, the corresponding solution to \eqref{equation1.2}  is given by 
\begin{equation*}
	u_j(t,x)=e^{-iN_j^2t}\sin{(N_j x)}.
\end{equation*}
Thus, 
\begin{align}\label{equ4.6}
	\int_{0}^T\sup_{x\in E}|u_j(t,x)|dt=T\cdot\sup_{x\in E}\left|\sin(N_j x)\right|.
\end{align}
We establish the critical decay property
\begin{align}\label{equ4.7}
	\sup\limits_{x\in E}\big|\sin(N_j x)\big|\to 0,\qquad\mbox{as }j\to\infty.
\end{align}
Indeed, for any $x\in E$ and  $j\in\mathbb{N}$, there exists some $m\in\mathbb{Z}$, such that
$x\in (\frac{\pi}{N_j}m,\frac{\pi}{N_j}m+\frac{\delta}{N_j^{1/\beta}}).$
Therefore,
$$0<N_j x-m\pi<\delta\cdot\frac{1}{N_j^{1/\beta-1}}<\frac{\pi}{2}.$$
Since $\beta<1$ (by \eqref{equ2.1.8}) and the sequence $N_j\to\infty$ as $j\to\infty$, we deduce that for all $x\in E$,
$$\big|\sin(N_j x)\big|<\big|\sin(\delta\cdot\frac{1}{N_j^{1/\beta-1}})\big|\leq \delta\cdot\frac{1}{N_j^{1/\beta-1}}\to 0,\quad j\to \infty.$$
Thus \eqref{equ4.7} holds.
Combining \eqref{eq-u_n-0-6-21}, \eqref{equ4.6} and \eqref{equ4.7} shows that inequality   \eqref{eq-heat} does not hold for initial data $u_{0,j}$. 
Therefore, the proof of Theorem \ref{thm3} is complete.\qed

\subsection{The proof of Theorem \ref{thm4}.}

In the following, we write $u(t,x)$ (with $(t,x)\in(0,\infty)\times\R$) for the solution of \eqref{equation1.2} with the initial condition $u(0,x)=u_0(x)$ over $\R$. We recall that for any $T>0$ and $u_0\in L^2(\R)$, the solution can be rewritten in the form (see e.g. in \cite{EKPV}), 
\begin{equation}\label{equ2.24}
	u(T,x)=\frac{1}{(2iT)^{1/2}}e^{ix^2/{4T}}\widehat{e^{i\xi^2/{4T}}u_0(\xi)}(x/{2T}),\quad x\in\R.
\end{equation}

We first prove statement $(i)$. By the conservation law for the  Schr\"{o}dinger equation, i.e., $\|u(T,\cdot;u_0)\|_{L^2}=\|u_0\|_{L^2}$ holds for any $T\ge 0$, it suffices to prove \eqref{equ4.2} for $S=0$.

Our argument relies on the following uncertainty principle due to Kovrizhkin \cite[Theorem 1.1]{Kov}, which  generalizes \cite[Theorem 2.1]{SVW}: 
If $\rho$ and $\tilde{\rho}$ are continuous non-increasing functions satisfying 
\begin{equation}\label{eq-thin-1024}
	\frac{C_1}{\tilde\rho(\frac{C_2}{\rho(t)})}\ge t,\qquad\;t> 0,
\end{equation}
for some absolute constants $C_1, C_2>0$, then there exist $\varepsilon > 0$ and $C > 0$ such that for any pair of $\varepsilon$-thin sets $E^c$ and $F^c$ with respect to $\rho$ and $\tilde\rho$, the following inequality holds:
\begin{equation}\label{eq-UP-W}
	\int_{\R}|f|^2 \leq C\left(\int_{E}|f(x)|^2 \d x + \int_{F}|\hat{f}(x)|^2 \d x\right),\qquad f\in L^2.
\end{equation}

To apply this  result, we  require the following two facts. 

$(i)$  If $E\subset\R$ is $\varepsilon$-thin with respect to $\rho$ defined in \eqref{eq-rho}, then for any fixed $r>0$, the scaled set $rE:=\{rx,\,\,x\in E\}$ is $\varepsilon$-thin with respect to 
\begin{equation}\label{eq-rho-tild}
	\tilde\rho(|x|)=\min(r,r^2|x|^{-1}).
\end{equation}
Indeed,  since
\begin{align}
	\frac{\tilde\rho(|x|)}{r}=
	\begin{cases}
		1,\quad\quad\;\mbox{if }|x|\le r,\\
		|\frac{x}{r}|^{-1},\quad\mbox{if }|x|\ge r.
	\end{cases}
\end{align}
Then from the definition of $\varepsilon$-thinness \eqref{eq-def-rho-1024}, we derive that
\begin{align*}
	|rE\cap B(x,\tilde\rho(|x|))|=r|E\cap B(\frac{x}{r},\frac{\tilde\rho(|x|)}{r})|\le r\cdot\varepsilon|B(\frac{x}{r},\frac{\tilde\rho(|x|)}{r})|=\varepsilon|B(x,\tilde\rho(|x|))|.
\end{align*}

$(ii)$ For $\rho$ and $ \tilde\rho$  given by \eqref{eq-rho} and \eqref{eq-rho-tild} respectively,  the following inequality holds
\begin{equation*}
	\frac{r}{\tilde\rho(\frac{r}{\rho(t)})}\ge t,\qquad\forall\;t> 0.
\end{equation*}
In fact, this can be verified by direct computation:
\begin{align*}
	\frac{r}{\tilde\rho(\frac{r}{\rho(t)})}=\frac{r}{\min\{r,r^2(\frac{r}{\min\{1,t^{-1}\}})^{-1}\}}=\frac{1}{\min\{1,\min\{1,t^{-1}\}\}}
	=\begin{cases}
		1,\quad\mbox{if } 0\le t\le 1,\\
		t,\,\quad\mbox{if } t>1.
	\end{cases}
\end{align*}

Based on these two facts,  we can now apply the  uncertainty principle \eqref{eq-UP-W} with $E=E_1$ and $F=E_2/2T=\{x/2T,\,\,x\in E_2\}$. Consequently,  for any $T>0$, there exists $C > 0$ such that
\begin{equation}\label{eq-UP-W-1024}
	\int_{\R}|f|^2 \leq C\left(\int_{E_1}|f(x)|^2 \d x + \int_{E_2/2T}|\hat{f}(x)|^2 \d x\right),\qquad f\in L^2.
\end{equation}

Take
$f(x)=e^{ix^2/{4T}}u_0(x)$ in \eqref{eq-UP-W-1024} and use the identity \eqref{equ2.24},
we obtain that
\begin{align}
	\|u_0\|_{L^2(\R)}^2&\leq  C\left(\int_{E_1}|u_0|^2 \d x + \int_{E_2}|u(T,x)|^2 \d x\right) \nonumber\\
	&= C\sum_{k\in\mathbb{Z}}  \left(\int_{E_1\cap[k,k+1]}|u_0|^2 \d x + \int_{E_2\cap[k,k+1]}|u(T,x)|^2 \d x\right)      \nonumber\\
	&\leq C\,\bigg(\big\|\big\{\sup_{E_1\cap[k,k+1]} |u_0|\big\}_{k\in\mathbb{Z}}\big\|_{l^2(\mathbb{Z})}^2+\big\|\big\{\sup_{E_2\cap[k,k+1]} |u(T,x)|\big\}_{k\in\mathbb{Z}}\big\|_{l^2(\mathbb{Z})}^2\bigg).  
\end{align}

Now  we establish statement $(ii)$.
For simplicity of exposition, we normalize the time parameters by setting $S=0$ and $T=1/2$ without loss of generality. The general case $T=S+\tau$ for arbitrary $\tau>0$ follows through scaling arguments and the fact $e^{it\Delta}$ is unitary.

Our construction begins with the following test functions
\begin{equation}\label{eq-sch-R-01}
	f_n(x)=\sin{(nx)}\left(e^{-(x+\pi n)^2/2}-e^{-(x-\pi n)^2/2}\right), \qquad n\in\mathbb{N}.  
\end{equation}
A direct computation yields 
\begin{align}\label{eq-f-n-01}
	\int_{\R}|f_n(x)|^2dx&=\int_{\R}\sin^2{(nx)}\left(e^{-(x+\pi n)^2}+e^{-(x-\pi n)^2}-2e^{-(x^2+(\pi n)^2)}\right)dx\nonumber\\
	&=\sqrt{\pi}\left(1-e^{-(\pi n)^2}\right)\left(1-e^{-n^2}\right)\to \sqrt{\pi},\quad \mbox{as}\,\,n\to\infty.
\end{align}
The Fourier transform of $f_n$ exhibits a structure analogous to $f_n(x)$. Specifically, we have
\begin{align}\label{eq-f-n-01-02}
	\hat{f_n}(\xi)
	&=\frac{1}{2i\sqrt{2\pi}}\int_{\R}e^{-ix\xi}(e^{inx}-e^{-inx})(e^{-(x+\pi n)^2/2}-e^{-(x-\pi n)^2/2})dx\nonumber\\
	&=(-1)^{n^2}\cdot\sin{(\pi n\xi)}\left(e^{-(\xi-n)^2/2}-e^{-(\xi+n)^2/2}\right).
\end{align}

Second, we define the sets $E_1$ and $E_2$ as follows. For $E_1$, we take $N_j$ as give by \eqref{equ4.3} and define the intervals 
\begin{equation*}
	J_{m,\delta}^{(1)}(j)=[\frac{\pi}{N_j}m,\, \,\frac{\pi}{N_j}m+\frac{\delta}{N_j^{1/\beta}}],
\end{equation*}
where $m\in\mathbb{Z}$, and $\beta$ is given by  \eqref{equ2.1.8}. The set $E_1$ is then constructed as
\begin{equation}\label{eq-E1-6-2}
	E_1=\bigcap_{j=1}^{\infty}\bigcup_{m\in\mathbb{Z}}J_{m,\delta}^{(1)}(j).
\end{equation}
Following the proof of Theorem \ref{thm1}, we deduce that
\begin{equation*}
	\dim_{\mathcal{H}}(E_1\cap[0,\pi])=\tilde{\alpha}.
\end{equation*}
Moreover, due to the uniform distribution of  $E_1$ on $\R$, this property holds uniformly across all translates:
$$\dim_{\mathcal{H}}(E_1\cap[x,x+\pi])=\tilde{\alpha},\qquad \mbox{for any }\,x\in\R.$$
Consequently, the Hausdorff capacity satisfies $\inf_{x\in\R}C^{\alpha}_{\mathcal{H}}(E_1\cap[x,x+\pi])>0$ for all $x\in\R$.

Similarly, for $E_2$, we take $N_j$ as given by \eqref{equ4.3}, with $\pi$ replaced by 1, and define
\begin{equation*}
	J_{m,\delta}^{(2)}(j)=[\frac{1}{N_j}m,\,\,\frac{1}{N_j}m+\frac{\delta}{N_j^{1/\beta}}],\quad m\in\mathbb{Z},
\end{equation*}
and the set $E_2$ is then given by
\begin{equation}\label{eq-E1-6-2-2}
	E_2=\bigcap_{j=1}^{\infty}\bigcup_{m\in\mathbb{Z}}J_{m,\delta}^{(2)}(j). 
\end{equation}
Analogously, we obtain
\begin{equation*}
	C^{\alpha}_{\mathcal{H}}(E_2\cap[x,x+1])>0,\qquad \mbox{for any }\,x\in\R.
\end{equation*}
Thus, $E_1$, $E_2$ are $\alpha$-thick sets.

Third, we establish the following
\begin{equation}\label{equ2.19}
	\left\|\left\{\sup_{E_1\cap[k,k+L]} |f_{N_j}(x)|\;\right\}_{k\in\mathbb{Z}}\right\|_{l^2(\mathbb{Z})}^2+\left\|\left\{\sup_{E_2\cap[k,k+L]} |\hat{f_{N_j}}(x)|\;\right\}_{k\in\mathbb{Z}}\right\|_{l^2(\mathbb{Z})}^2\to 0,\quad j\to\infty.
\end{equation}
For brevity, we focus on proving 
\begin{equation}\label{equ2.20}
	\left\|\left\{\sup_{E_1\cap[k,k+L]} |f_{N_j}(x)|\;\right\}_{k\in\mathbb{Z}}\right\|_{l^2(\mathbb{Z})}^2\to 0,\quad j\to\infty.
\end{equation}
The proof for $E_2$ follows in the same way due to \eqref{eq-f-n-01-02}.
From the construction  of $E_1$ (see \eqref{eq-E1-6-2}), it follows that for any $j\in\mathbb{N}$ and $x\in E_1$, there exists some $m\in\mathbb{Z}$, such that
\begin{equation*}
	x\in [\frac{\pi}{N_j}m,\frac{\pi}{N_j}m+\frac{\delta}{N_j^{1/\beta}}].
\end{equation*}
This implies 
$\pi m\leq N_jx\leq \pi m+\frac{\delta}{N_j^{1/\beta-1}},$
which in turn yields
\begin{equation*}
	|\sin{(N_jx)}|\leq |\sin{(\frac{\delta}{N_j^{1/\beta-1}})}|\leq \frac{\delta}{N_j^{1/\beta-1}}.
\end{equation*}
Then, the $l^2$ norm can be bounded as 
\begin{align}\label{equ2.21}
	\quad\|\{\sup_{E_1\cap[k,k+1]} |f_{N_j}(x)|\;\}_{k\in\mathbb{Z}}\|_{l^2(\mathbb{Z})}^2\notag 
	&=\sum_{k\in\mathbb{Z}} \sup_{E_1\cap[k,k+1]}|\sin{(N_jx)}|^2|e^{-(x+\pi N_j)^2/2}-e^{-(x-\pi N_j)^2/2}|^2\notag \\
	&\leq \sum_{k\in\mathbb{Z}} \left(\sup_{E_1\cap[k,k+1]}\sin^2{(N_jx)}\cdot\sup_{[k,k+1]}\left(e^{-(x+\pi N_j)^2}+e^{-(x-\pi N_j)^2}-2e^{-(\pi N_j)^2}e^{-x^2}\right)\right)\notag \\
	&\leq \frac{\delta^2}{N_j^{2/\beta-2}}\cdot\sum_{k\in\mathbb{Z}}\left(\sup_{[k,k+1]}\left(e^{-(x+\pi N_j)^2}+e^{-(x-\pi N_j)^2}\right)\right).
\end{align}
Note that the summation of Gaussian terms admits the following uniform bound
\begin{align}\label{equ2.23}
	\quad\sum_{k\in\mathbb{Z}}\left(\sup_{[k,k+1]}\left(e^{-(x+\pi N_j)^2}+e^{-(x-\pi N_j)^2}\right)\right)
	&\leq 2\left(\sum_{k\in\mathbb{Z}}e^{-k^2}+3\right)
	\leq 2(\sqrt{\pi}+3).
\end{align}
Since $\beta<1$ by hypothesis and $N_j\to\infty$ as $j\to\infty$,
the desired convergence in (\ref{equ2.20}) follows from \eqref{equ2.21} and \eqref{equ2.23}. 

Finally, we complete the proof of Theorem \ref{thm4}.
Define the initial data
\begin{equation*}
	\tilde u_{0,j}(x)=e^{-ix^2/2}f_{N_j}(x),\quad j\in\mathbb{N}.
\end{equation*}
By \eqref{equ2.24}, the corresponding solution at time $T=\frac12$ becomes
\begin{equation*}
	\tilde u_j(\frac{1}{2},x)=\frac{1}{i^{1/2}}e^{ix^2/2}\widehat{e^{i\xi^2/2}\tilde u_{0,j}(\xi)}(x)=\frac{1}{i^{1/2}}e^{ix^2/2}\hat f_{N_j}(x).
\end{equation*}
From (\ref{equ2.19}), we obtain
\begin{equation}\label{equ2.25}
	\|\{\sup_{E_1\cap[k,k+1]} |\tilde u_{0,j}(x)|\;\}_{k\in\mathbb{Z}}\|_{l^2(\mathbb{Z})}^2+\|\{\sup_{E_2\cap[k,k+1]} |\tilde u_j(\frac{1}{2},x )|\;\}_{k\in\mathbb{Z}}\|_{l^2(\mathbb{Z})}^2\to 0, \quad j\to\infty.
\end{equation}
On the other hand,  by \eqref{eq-f-n-01}, we have
\begin{equation}\label{equ2.27}
	\|\tilde u_{0,j}(x)\|_{L^2(\R)}^2=\| f_{N_j}(x)\|_{L^2(\R)}^2\to\sqrt{\pi},\quad j\to\infty.
\end{equation}
Combining \eqref{equ2.25} and \eqref{equ2.27} yields the proof of  statement $(ii)$.

Therefore the proof of Theorem \ref{thm4} is complete. \qed

\section*{Acknowledgments}
S. Huang was supported by the National Natural Science Foundation of China under grants 12171178 and 12171442 and the Guangdong Basic and Applied Basic Research Foundation (2026B1515020075).


\end{document}